\documentclass[review,onefignum,onetabnum]{siamart171218}

\usepackage{lipsum}
\usepackage{amsfonts}
\usepackage{graphicx}
\usepackage{epstopdf}
\usepackage{algorithmic}
\ifpdf
  \DeclareGraphicsExtensions{.eps,.pdf,.png,.jpg}
\else
  \DeclareGraphicsExtensions{.eps}
\fi

\newsiamremark{remark}{Remark}
\newsiamremark{hypothesis}{Hypothesis}
\crefname{hypothesis}{Hypothesis}{Hypotheses}
\newsiamthm{claim}{Claim}

\headers{Adaptive Step Selection for a Filtered Implicit Method}{S. M. McGovern}

\title{Adaptive Step Selection for a Filtered Implicit Method}

\author{Stephen M. McGovern\thanks{University of Pittsburgh, Department of Mathematics (\email{stm114@pitt.edu}})
}
\usepackage{amsopn}

\makeatletter
\newcommand*{\addFileDependency}[1]{
  \typeout{(#1)}
  \@addtofilelist{#1}
  \IfFileExists{#1}{}{\typeout{No file #1.}}
}
\makeatother

\newcommand*{\myexternaldocument}[1]{%
    \externaldocument{#1}%
    \addFileDependency{#1.tex}%
    \addFileDependency{#1.aux}%
}

\ifpdf
\hypersetup{
  pdftitle={Adaptive Step Selection for a Filtered Implicit Method},
  pdfauthor={S. M. McGovern}
}
\fi

\myexternaldocument{ex_supplement}

\newtheorem{prop}{Proposition}
\begin{document}
\nolinenumbers
\maketitle

\begin{abstract}
    Pre-filtering and post-filtering steps can be added to many of the traditional numerical methods to generate new, higher order methods with strong stability properties. Presented in this paper are a variable step pre-filter and post-filter that allow adaptive time stepping for a filtered method based on Implicit Euler from \cite{decaria:2022}.
\end{abstract}



\section{Introduction}
\label{sec:intro}
To address needs in industry and research, enhanced numerical methods can be generated by the use of filters \cite{ditkowski:2015}. They can be applied before and after a ``solve step'' in many of the traditional numerical methods to create new, higher order methods \cite{ditkowski:2015}. The filtering steps tend to be only a few lines of code each, which is advantageous from an implementation standpoint. Additionally, the new filtered methods have favorable accuracy and stability properties despite the relative simplicity of the code changes. Ideas of this nature are presented in 
\cite{decaria:2022} \cite{ditkowski:2015} \cite{ditkowski:2017} \cite{guzel:2018}.

Complex codes based on implicit methods in particular are common. They appear frequently in modern and legacy code bases. It has been shown that their accuracy can be increased through the use of filters while maintaining stability \cite{decaria:2022}. The focus of this paper is adaptive step selection for two novel filtered methods based on Implicit Euler from \cite{decaria:2022} that were presented with a constant step. The two methods are a natural embedded pair which makes them amenable to adaptive stepping. However, no variable step extension of the methods exists prior to this paper. We develop the variable step extension herein.

For reference, the first constant step method from \cite{decaria:2022} is a second order Implicit Euler-based method with a pre-filter. The other constant step method is a third order Implicit Euler-based method with the same pre-filter, and then a post-filter added after the Implicit Euler solve. The constant step methods read as follows:
\begin{align*}
    &\text{Step 1) Pre-filter} && \Tilde{y}_n = y_n - \frac{1}{2}(y_n - 2y_{n-1} + y_{n-2}) ,\\[2ex]
    &\text{Step 2) IE Solve} && \frac{y_{n+1} - \Tilde{y}_n}{k} = f(t_{n+1}, y_{n+1}) ,\\[2ex]
    &\text{Step 3) Post-filter} && y_{n+1}^{3rd} = y_{n+1} - \frac{5}{11}(y_{n+1} - 3y_n + 3y_{n-1} - y_{n-2}).
\end{align*}

Completing Step 1) and Step 2) above yields the second order method. We will refer to it as IE-Pre-2. It is A-stable and L-stable \cite{decaria:2022}. Once Step 3) is added for post-filtering, then completing all three steps above yields the third order method. This third order method will be referred to as IE-Pre-Post-3. A-stability is lost after adding the post-filter step. However, IE-Pre-Post-3 is still A($\alpha$) stable with $\alpha \approx 71.51^{\circ}$ (Section 4.1.3 of \cite{decaria:2022}). 

The higher order method shares the same first two steps with the lower order method. This fact makes the two methods a natural embedded pair. The variable step extension derived in \Cref{sec:main} is
\begin{align*}
    &\text{Step 1) Pre-filter} && \Tilde{y}_n = y_n 
    \\[2ex]& &&- \frac{\alpha_n}{2} \left(\frac{2k_{n-2}}{k_{n-1}+k_{n-2}}y_{n} - 2y_{n-1} + \frac{2k_{n-1}}{k_{n-1}+k_{n-2}}y_{n-2} \right),\\[2ex]
    &\text{Step 2) IE Solve} && \frac{y_{n+1}-\Tilde{y}_n}{k_n} = f(t_{n+1}, y_{n+1}),\\[2ex]
    &\text{Step 3) Post-filter} && y^{3rd}_{n+1} = y_{n+1}
    \\[2ex]& &&- \beta_n\left(\frac{2k_{n-1}}{k_n+k_{n-1}}y_{n+1} - 2y_n + \frac{2k_n}{k_n+k_{n-1}}y_{n-1}\right.
    \\[2ex]& &&- \left.\left(\frac{2k_{n-2}}{k_{n-1}+k_{n-2}}y_{n} - 2y_{n-1} + \frac{2k_{n-1}}{k_{n-1}+k_{n-2}}y_{n-2}\right)\right).
\end{align*}
Above, $\alpha_n$ and $\beta_n$ are the pre-filter and post-filter coefficients, respectively. It is shown how to choose these filters for the requisite orders in \Cref{sec:main}. The variables $k_n$, $k_{n-1}$ and $k_{n-2}$ are steps between values of $t$.

Denote the value computed from Step 1) and Step 2) of the method as $y^{2nd}_{n+1}$ and denote the value computed from Step 1), Step 2) and Step 3) of the complete method as $y^{3rd}_{n+1}$. Then we can use
\begin{align*}
    EST = |y^{3rd}_{n+1} - y_{n+1}^{2nd}|
\end{align*}
as an estimator of the error at Step 2).

We use this error estimator, along with techniques to minimize discrete curvature from \cite{guzel:2018}, to derive the filters and construct the variable time step method above. The new method we call Filtered-IE23. This adaptive numerical method is implemented with the simplest controller, halving and doubling, since the focus is on the estimator. The paper is organized as follows. Background is presented in \Cref{sec:background}, the main results concerning the derivation of the new adaptive pre-filter and post-filter are in \Cref{sec:main}, the algorithm for a halving and doubling scheme is given in \Cref{sec:alg}, numerical results are in \Cref{sec:numerics}, and the conclusions follow in \Cref{sec:conclusions}.

\section{Background}
\label{sec:background}

We follow the definition of discrete curvature presented in \cite{guzel:2018}, pages 7-8. Briefly, we find the Lagrange interpolating polynomial for $y_{n+1}$, $y_n$ and $y_{n-1}$. Call this quadratic interpolant $\phi(t)$. We have
\begin{align*}
    \phi (t) &= y_{n+1} \ell_{n+1}(t) + y_{n} \ell_{n}(t) + y_{n-1} \ell_{n-1}(t)\\[2ex]
    &= y_{n+1} \frac{(t-t_{n-1})(t-t_n)}{(t_{n+1}-t_{n-1})(t_{n+1}-t_n)} 
    + y_{n} \frac{(t-t_{n-1})(t-t_{n+1})}{(t_{n}-t_{n-1})(t_{n}-t_{n+1})} \\[2ex]
    &+ y_{n-1} \frac{(t-t_n)(t-t_{n+1})}{(t_{n-1}-t_n)(t_{n-1}-t_{n+1})}.
\end{align*}
Differentiate twice with respect to $t$ to find the discrete second difference, which is
\begin{align*}
    \phi''(t) &= \frac{2}{(t_{n+1}-t_{n-1})(t_{n+1}-t_n)} y_{n+1} 
    + \frac{2}{(t_{n}-t_{n-1})(t_{n}-t_{n+1})} y_{n} \\[2ex]
    &+\frac{2}{(t_{n-1}-t_n)(t_{n-1}-t_{n+1})} y_{n-1}.
\end{align*}
Denote $k_n$ as the time step between $t_n$ and $t_{n+1}$ and denote $k_{n-1}$ as the time step from $t_{n-1}$ to $t_n$. Then we have
\begin{align*}
    \phi''(t) &= \frac{2}{(k_{n-1}+k_n)k_n} y_{n+1} 
    + \frac{2}{k_{n-1}(-k_n)} y_{n} + \frac{2}{(-k_{n-1})(-(k_n+k_{n-1}))} y_{n-1}.
\end{align*}
Curvature $\kappa_n$, as defined in \cite{guzel:2018}, is then $\phi''(t)$ scaled by $k_{n-1}k_n$. So
\begin{align*}
    \kappa_n &= k_{n-1}k_n\phi'' = \frac{2k_{n-1}}{k_n+k_{n-1}}y_{n+1} - 2y_n + \frac{2k_n}{k_n+k_{n-1}}y_{n-1}.
\end{align*}
Similar ideas are presented in \cite{kalnay:2003} \cite{williams:2011}.

The underlying idea for both the constant and variable step method is
\begin{align*}
    &\text{Step 1) Pre-filter} && \Tilde{y}_n = y_n - \frac{\alpha_n}{2}\kappa_{n-1}, \\[2ex]
    &\text{Step 2) IE Solve} && \frac{y_{n+1}-\Tilde{y}_n}{k_n} = f(t_{n+1}, y_{n+1}),\\[2ex]
    &\text{Step 3) Post-filter} && y^{3rd}_{n+1} = y_{n+1} - \beta_n(\kappa_{n} - \kappa_{n-1}).
\end{align*}
In both cases, $\alpha_n$ and $\beta_n$ are chosen accordingly for 2nd and 3rd order. The pre-filter is the Robert-Asselin filter \cite{rasselin:1972} \cite{robert:1969}. The post-filter is a higher-order Robert-Asselin-type time filter, presented in \cite{trenchea}. Note that in the constant step case, we have $k_n = k_{n-1} = k_{n-2} = k$ and thus the curvatures at steps $n-1$ and $n$ simplify to
\begin{align*}
    \kappa_{n-1} &= y_n - 2y_{n-1} + y_{n-2}\,\,\,\text{and}\,\,\,\kappa_n = y_{n+1} - 2y_n + y_{n-1}.
\end{align*}
Then $\kappa_{n-1}$ scaled by the pre-filter coefficient is found in Step 1) of the constant step method in ~\Cref{sec:intro}. Likewise,
\begin{align*}
    \kappa_n - \kappa_{n-1} = y_{n+1} - 3y_n +3y_{n-1} - y_{n-2}
\end{align*}
is scaled by the post-filter coefficient in Step 3) of the constant step method.

\section{Main Results}
\label{sec:main}

Now that the background is in place, we show how to derive the adaptive pre-filter and post-filter for the desired orders. We use direct Taylor expansions to determine the pre-filter $\alpha_n$ for second order and an alternative method to find the post-filter $\beta_n$ for third order.

\subsection{Adaptive Step Pre-filter}
Consider the first two steps of adaptive method Filtered-IE23 from \Cref{sec:intro}.
We must choose $\alpha_n$ so that Step 1) and Step 2) yield a second order method.
\begin{prop}The value for the pre-filter $\alpha_n$ which gives a second order adaptive method after completing Step 1) and Step 2) is
\begin{align*}
    \alpha_n = \frac{k_n^2}{k_{n-1}k_{n-2}}.
\end{align*}
\end{prop}
\begin{proof}
The strategy is to construct the equivalent multistep method, Taylor expand around each of the terms and choose $\alpha_n$ accordingly. Consider Step 1) and Step 2):
\begin{align*}
    &\Tilde{y}_n = y_n - \frac{\alpha_n}{2}\left(\frac{2k_{n-2}}{k_{n-1}+k_{n-2}}y_{n} - 2y_{n-1} + \frac{2k_{n-1}}{k_{n-1}+k_{n-2}}y_{n-2}\right) ,\\[2ex]
    &\frac{y_{n+1}-\Tilde{y}_n}{k_n} = f(t_{n+1}, y_{n+1}).
\end{align*}
Substitute the pre-filtered value from Step 1) into the Implicit Euler solve in Step 2), gather terms and simplify:
\begin{align*}
    &y_{n+1} - \left( y_n - \frac{\alpha_n}{2}\left(\frac{2k_{n-2}}{k_{n-1}+k_{n-2}}y_{n} - 2y_{n-1} + \frac{2k_{n-1}}{k_{n-1}+k_{n-2}}y_{n-2}\right) \right) \\[2ex]&= k_n f(t_{n+1},y_{n+1}),\,\,\,\text{or}\\[2ex]
    &y_{n+1} + \frac{(\alpha_n - 1)k_{n-2} - k_{n-1}}{k_{n-1}+k_{n-2}}y_{n} - \alpha_n y_{n-1} + \frac{\alpha_n k_{n-1}}{k_{n-1}+k_{n-2}}y_{n-2} - k_n (y')_{n+1} \\[2ex]&= 0.
\end{align*}
Taylor expand around each of the terms:
\begin{align*}
    y_{n+1} &= y_{n+1},
\\[2ex] 
    \frac{(\alpha_n - 1)k_{n-2} - k_{n-1}}{k_{n-1}+k_{n-2}} y_n &= 
       \frac{(\alpha_n - 1)k_{n-2} - k_{n-1}}{k_{n-1}+k_{n-2}} y_{n+1}
       - \frac{(\alpha_n - 1)k_{n-2} - k_{n-1}}{k_{n-1}+k_{n-2}}k_n y' 
       \\[2ex]   &+ \frac{((\alpha_n - 1)k_{n-2} - k_{n-1})}{k_{n-1}+k_{n-2}}\frac{k_n^2}{2}y'' +\mathcal{O}(k_n^3),
\\[2ex]  
    -\alpha_n y_{n-1} &= -\alpha_n y_{n+1} + \alpha_n (k_n + k_{n-1})y' 
\\[2ex]    
    &- \alpha_n \frac{(k_n + k_{n-1})^2}{2}y'' +\mathcal{O}((k_n+k_{n-1})^3),
\\[2ex]  
    \frac{\alpha_n k_{n-1}}{k_{n-1}+k_{n-2}} y_{n-2} &=
        \frac{\alpha_n k_{n-1}}{k_{n-1}+k_{n-2}} y_{n+1} 
        - \frac{\alpha_n k_{n-1}}{k_{n-1}+k_{n-2}} (k_n + k_{n-1} + k_{n-2})y' 
        \\[2ex]  &+ \frac{\alpha_n k_{n-1}}{k_{n-1}+k_{n-2}} \frac{(k_n + k_{n-1} + k_{n-2})^2}{2} y'' 
        \\[2ex] &+\mathcal{O}((k_n + k_{n-1} + k_{n-2})^3),
\\[2ex]  -k_n (y')_{n+1} &=  -k_n y'.
\end{align*}
For a second order method, $\alpha_n$ must be chosen so that the $y_{n+1}$ , $y'$ and $y''$ terms sum to zero. Sum the coefficients on the $y_{n+1}$ terms and we have:
\begin{align*}
    1 + \frac{(\alpha_n - 1)k_{n-2} - k_{n-1}}{k_{n-1}+k_{n-2}} + (-\alpha_n) + \frac{\alpha_n k_{n-1}}{k_{n-1}+k_{n-2}} = 0.
\end{align*}
This is zero regardless of the choice of $\alpha_n$. Similarly, sum the coefficients on the $y'$ terms:
\begin{align*}
    &- \frac{(\alpha_n - 1)k_{n-2} - k_{n-1}}{k_{n-1}+k_{n-2}}k_n + \alpha_n (k_n + k_{n-1}) 
    \\[2ex]&-\frac{\alpha_n k_{n-1}}{k_{n-1}+k_{n-2}} (k_n + k_{n-1} + k_{n-2}) - k_n  = 0.
\end{align*}
This is also trivially zero. Finally, sum the coefficients on the $y''$ terms:
\begin{align*}
    &\frac{((\alpha_n - 1)k_{n-2} - k_{n-1})}{k_{n-1}+k_{n-2}}\frac{k_n^2}{2}
    + -\alpha_n \frac{(k_n + k_{n-1})^2}{2}
    \\[2ex]&+ \frac{\alpha_n k_{n-1}}{k_{n-1}+k_{n-2}} \frac{(k_n + k_{n-1} + k_{n-2})^2}{2}
    \\[2ex]
    &= \frac{\alpha_n k_n^2 k_{n-2} - k_n^2 k_{n-2} - k_n^2 k_{n-1}}{2(k_{n-1}+k_{n-2})}
    + \frac{-\alpha_n k_n^2 k_{n-1} - \alpha_n k_n^2 k_{n-2}^2 -2 \alpha_n k_n k_{n-1}^2}{2(k_{n-1}+k_{n-2})}
    \\[2ex]&+ \frac{- 2 \alpha_n k_n k_{n-1} k_{n-2} - \alpha_n k_{n-1}^3 - \alpha_n k_{n-1}^2 k_{n-2} +\alpha_n k_n^2 k_{n-1} + 2\alpha_n k_n k_{n-1}^2}{2(k_{n-1}+k_{n-2})}
    \\[2ex]&+ \frac{2 \alpha_n k_n k_{n-1} k_{n-2} + \alpha_n k_{n-1}^3 + 2\alpha_n k_{n-1}^2 k_{n-2} + \alpha_n k_{n-1} k_{n-2}^2}{2(k_{n-1}+k_{n-2})}.
\end{align*}
We require this to sum to zero:
\begin{align*}
    0&=-k_n^2 k_{n-2} - k_n^2 k_{n-1} - \alpha_n k_{n-1}^2 k_{n-2} + 2\alpha_n k_{n-1}^2 k_{n-2} + \alpha_n k_{n-1} k_{n-2}^2 \\[2ex]
    &=-k_n^2 k_{n-2} - k_n^2 k_{n-1} + \alpha_n k_{n-1}^2 k_{n-2} + \alpha_n k_{n-1} k_{n-2}^2 \\[2ex]&\Longleftrightarrow
    \alpha_n k_{n-1}k_{n-2} (k_{n-1} + k_{n-2}) = k_n^2 ( k_{n-1} + k_{n-2} )
    \\[2ex]&\Longleftrightarrow \alpha_n = \frac{k_n^2}{k_{n-1}k_{n-2}}.
\end{align*}
\end{proof}
With this choice of $\alpha_n$, the local truncation error is $\mathcal{O}((k_n + k_{n-1} + k_{n-2})^3)$ and so the method is second order \cite{ascher:1998} \cite{Griffiths:2010}. Note that in the constant step case, i.e. when $k_n = k_{n-1} = k_{n-2}$, $\alpha_n$ reduces to 1. This is consistent with the pre-filter coefficient in Step 1) of the constant step method given in ~\Cref{sec:intro}.

\subsection{Adaptive Step Post-filter}
Now that $\alpha_n$ has been found, the variable step method so far is:
\begin{align*}
    &\text{Step 1) Pre-filter} && \Tilde{y}_n = y_n 
    \\[2ex]& &&- \frac{k^2_n}{2k_{n-2}k_{n-1}} \left(\frac{2k_{n-2}}{k_{n-1}+k_{n-2}}y_{n} - 2y_{n-1}\right.\\[2ex] 
    & &&+ \left. \frac{2k_{n-1}}{k_{n-1}+k_{n-2}}y_{n-2} \right),\\[2ex]
    &\text{Step 2) IE Solve} && \frac{y_{n+1}-\Tilde{y}_n}{k_n} = f(t_{n+1}, y_{n+1}),\\[2ex]
    &\text{Step 3) Post-filter} && y^{3rd}_{n+1} = y_{n+1} - \beta_n(\kappa_{n} - \kappa_{n-1}).
\end{align*}
Given $\alpha_n$, we must choose the post-filter $\beta_n$ so the complete, adaptive step method is third order.
\begin{prop}
    The value for the post-filter $\beta_n$ which gives a third order adaptive method is
\begin{align*}
    \beta_n &= \frac{\beta_{n1}}{\beta_{n2}},\,\text{where}\\[2ex]
    \beta_{n1} &= -k_n^2 \left(k_{n-1}+k_n\right) \left(k_{n-2}+2 \left(k_{n-1}+k_n\right)\right),
    \\[2ex]
    \beta_{n2} &= 2 k_{n-1} (2 \left(k_{n-1}+k_n\right) k_{n-2}^2+\left(k_{n-1}^2-5 k_n k_{n-1}-7 k_n^2\right) k_{n-2}\\[2ex]
    &+3 k_{n-3} \left(k_{n-2}-k_n\right) \left(k_{n-1}+k_n\right)-2 k_{n-1} k_n \left(k_{n-1}+k_n\right)).
\end{align*}
\end{prop}
\begin{proof}
To find $\beta_n$, we use an alternative method to the Taylor expansions used in the previous subsection. We require the method to be exact on $y = t^3$. For simplicity and space considerations, use $n=3$. Then we'll work with $y_1$, $y_2$, $y_3$, $y_4$ in the method and
\begin{align*}
    y_1 &= t_1^3, && t_1 = k_0, \\[2ex]
    y_2 &= t_2^3, && t_2 = k_0 + k_1, \\[2ex]
    y_3 &= t_3^3, && t_3 = k_0 + k_1 + k_2, \\[2ex]
    y_4 &= t_4^3, && t_4 = k_0 + k_1 + k_2 + k_3.
\end{align*}
Then the post-filtered $y_{4}$, call it $y_4^{3rd}$, is
\begin{align*}
    y_4^{3rd}&=y_4-\beta  \left(\kappa _3-\kappa _2\right) \\[2ex]
    &= \frac{2 \beta  k_2 }{k_1+k_2}y_1 - \beta  \left(\frac{2 k_3}{k_2+k_3}+2\right) y_2 \\[2ex] &- \beta  \left(-\frac{2 k_1}{k_1+k_2}-2\right) y_3+ \left(1-\frac{2 \beta  k_2}{k_2+k_3}\right) y_4.
\end{align*}
Denote $y_4$ as $y_4^*$ and denote $y_{4}^{3rd}$ as $y_4$:
\begin{align*}
    y_4 &= \frac{2 \beta  k_2 }{k_1+k_2}y_1 - \beta  \left(\frac{2 k_3}{k_2+k_3}+2\right) y_2 
    \\[2ex] &- \beta  \left(-\frac{2 k_1}{k_1+k_2}-2\right) y_3+ \left(1-\frac{2 \beta  k_2}{k_2+k_3}\right) y_4^*.
\end{align*}
We solve for $y_4^*$ and plug it into the solve step of the method along with the pre-filter step:
\begin{align*}
y_4^* &= \frac{2 \beta  k_2 }{\left(k_1+k_2\right) \left(\frac{2 \beta  k_2}{k_2+k_3}-1\right)}y_1-\frac{\beta  \left(\frac{2 k_3}{k_2+k_3}+2\right) }{\frac{2 \beta  k_2}{k_2+k_3}-1}y_2
\\[2ex]&-\frac{\beta  \left(-\frac{2 k_1}{k_1+k_2}-2\right) }{\frac{2 \beta  k_2}{k_2+k_3}-1}y_3-\frac{1}{\frac{2 \beta  k_2}{k_2+k_3}-1}y_{4},\\[2ex]
\Tilde{y}_3 &= y_3-\frac{k_3^2}{2 k_1 k_2} \left(\frac{2 k_2 }{k_1+k_2}y_1 -2 y_2 +\frac{2 k_1}{k_1+k_2} y_3\right). 
\end{align*}
The method with $y_1$, $y_2$, $y_3$, $y_4$ defined previously is
\begin{align*}
    &y_4^* - \Tilde{y}_3 = k_3 (y_4^*)'\\[2ex]
        \Longleftrightarrow
        &\frac{ 2 \beta  k_2}{\left(k_1+k_2\right) \left(\frac{2 \beta  k_2}{k_2+k_3}-1\right)} k_0^3 
        -\frac{ \beta  \left(\frac{2 k_3}{k_2+k_3}+2\right)}{\frac{2 \beta  k_2}{k_2+k_3}-1} \left(k_0+k_1\right){}^3
        \\[2ex]&-\frac{\beta  \left(-\frac{2 k_1}{k_1+k_2}-2\right)}{\frac{2 \beta  k_2}{k_2+k_3}-1} \left(k_0+k_1+k_2\right){}^3 
        -\frac{1}{\frac{2 \beta  k_2}{k_2+k_3}-1} \left(k_0+k_1+k_2+k_3\right){}^3 \\[2ex]
        &-\left(k_0+k_1+k_2\right){}^3\\[2ex]
        &+\frac{k_3^2}{2 k_1 k_2} \left(\frac{\left(2 k_2\right) k_0^3}{k_1+k_2}-2 \left(k_0+k_1\right){}^3+\frac{\left(2 k_1\right) \left(k_0+k_1+k_2\right){}^3}{k_1+k_2}\right) \\[2ex]
        &= k_3 \left( \frac{2 \beta  k_2 }{\left(k_1+k_2\right) \left(\frac{2 \beta  k_2}{k_2+k_3}-1\right)} 3k_0^2 -\frac{\beta  \left(\frac{2 k_3}{k_2+k_3}+2\right) }{\frac{2 \beta  k_2}{k_2+k_3}-1} 3(k_0 + k_1)^2 \right.
        \\[2ex]&\left. -\frac{\beta  \left(-\frac{2 k_1}{k_1+k_2}-2\right) }{\frac{2 \beta  k_2}{k_2+k_3}-1} 3(k_0 + k_1 + k_2)^2 -\frac{1}{\frac{2 \beta  k_2}{k_2+k_3}-1}3(k_0 + k_1 + k_2 + k_3)^2\right).
\end{align*}
Solve the equation above for $\beta$. Due to space considerations, define $\beta = \frac{\beta_1}{\beta_2}$. We find that
\begin{align*}
    \beta_1 &= -k_3^2 \left(k_2+k_3\right) \left(k_1+2 \left(k_2+k_3\right)\right),\\[2ex]
    \beta_2 &= 
    2 k_2 (2 \left(k_2+k_3\right) k_1^2+\left(k_2^2-5 k_3 k_2-7 k_3^2\right) k_1\\[2ex]
    &+3 k_0 \left(k_1-k_3\right) \left(k_2+k_3\right)-2 k_2 k_3 (k_2+k_3)).
\end{align*}
Thus, in general, we choose the adaptive post-filter to be $\beta_n = \frac{\beta_{n1}}{\beta_{n2}}$ where
\begin{align*}
    \beta_{n1} &= -k_n^2 \left(k_{n-1}+k_n\right) \left(k_{n-2}+2 \left(k_{n-1}+k_n\right)\right),
    \\[2ex]
    \beta_{n2} &= 2 k_{n-1} (2 \left(k_{n-1}+k_n\right) k_{n-2}^2+\left(k_{n-1}^2-5 k_n k_{n-1}-7 k_n^2\right) k_{n-2}\\[2ex]
    &+3 k_{n-3} \left(k_{n-2}-k_n\right) \left(k_{n-1}+k_n\right)-2 k_{n-1} k_n \left(k_{n-1}+k_n\right)).
\end{align*}
\end{proof}
The adaptive method is third order with this choice of $\beta_n$ and the given value of $\alpha_n$ in the previous subsection. The technique used in this proof is also used in the appendices of \cite{mcgovern:2022} for alternative derivations of the constant time step filters and the adaptive step pre-filter. Note that in the constant step case, $\beta_n$ reduces to match the constant step post-filter.
\begin{lemma}
The adaptive post-filter $\beta_n$ simplifies to $\frac{5}{11}$ when $k_{n-3} = k_{n-2} = k_{n-1} = k_{n} = k$.
\end{lemma}
\begin{proof} Suppose $k_{n-3} = k_{n-2} = k_{n-1} = k_{n} = k$. Then
\begin{align*}
    \beta_{n1} &= -k^2(k+k)(k+2(k+k)) = -10k^4,\\[2ex]
    \beta_{n2} &= 2k(2(k+k)k^2 + (k^2-5k\cdot k - 7k^2)k\\[2ex]
    &+ 3k(k-k)(k+k)-2k\cdot k (k +k)\\[2ex]
    &= 2k( 4k^3 -11k^3 + 0 -4k^3) = -22k^4,\,\,\,\text{and so}\\[2ex]
    \beta_n &= \frac{-10k^4}{-22k^4} = \frac{5}{11}.
\end{align*}
\end{proof}
This is consistent with the post-filter coefficient in Step 3) of the constant step method IE-Pre-Post-3 given in \Cref{sec:intro}.

\section{Algorithm}
\label{sec:alg}
We provide the pseudocode for a naive halving and doubling algorithm in \Cref{alg:halvingdoubling} listed below. The simplest controller was chosen because this paper focuses on the estimator. 

\begin{algorithm}
\caption{Halving and Doubling Adaptive Time Step}
\tiny
\label{alg:halvingdoubling}
\begin{algorithmic}
\STATE{Define TOL := Given user tolerance, dt := Given initial time step}
\STATE{Define a := Given lower time bound}
\STATE{Define b := Given upper time bound}
\STATE{Initialize t := Time array, y := Y value array, e := Error array}
\STATE{$k_n = dt$}
\STATE{RK3 for the first 3 steps (3 values are needed to start filtering)}
\WHILE{$t_n < b$}
\STATE{$k_{n-1} = t_n - t_{n-1}$}
\STATE{$k_{n-2} = t_{n-1} - t_{n-2}$}
\STATE{$k_{n-3} = t_{n-2} - t_{n-3}$}
\WHILE{Halving and Doubling}
\STATE{Compute pre-filter: $\alpha_n$, curvature $\kappa_{n-1}$}
\STATE{$\Tilde{y}_n = y_n - \frac{\alpha_n}{2}\kappa_{n-1}$}
\STATE{Implicit Euler Solve: $\frac{y_{n+1} - \Tilde{y}_n}{k_n} = f(t_{n+1}  y_{n+1})$}
\STATE{$y_{n+1}^{2nd} = y_{n+1}$}
\STATE{Compute post-filter: $\beta_n$, curvature $\kappa_{n}$}
\STATE{$y^{3rd}_{n+1} = y_{n+1} - \beta_n(\kappa_{n} - \kappa_{n-1}$)}
\STATE{Set Error Estimator: EST := $|y^{3rd}_{n+1} - y_{n+1}^{2nd}|$}

\IF{TOL $\cdot \,\,k_n < $ EST}
\STATE{$k_n = k_n / 2.0$}
\ELSIF{EST $<$ TOL $\cdot \,\,k_n / 2^5$ }
\STATE{$k_n = k_2 \cdot 2$}
\STATE{Break Inner Loop}
\ELSE
\STATE{Continue}
\ENDIF
\ENDWHILE
\STATE{Append t, y, e values}
\STATE{Continue Outer Loop}
\ENDWHILE
\RETURN $t, y, e$
\end{algorithmic}
\end{algorithm}

\section{Numerical results}
\label{sec:numerics}
In this section we numerically test the adaptive method. First, Filtered-IE23 is implemented to solve the model problem to compare with the results of the constant step method IE-Pre-Post-3. We then run the method to solve a quasi-periodic ODE, a stiff nonautonomous problem, and the van der Pol Oscillator.

\subsection{Model Problem}
To establish a numerical baseline, we solve the canonical test problem with Filtered-IE23. The problem statement is
\begin{align*}
    y' &= \lambda y,\\
    y(0) &= 1.
\end{align*}
For simplicity, set $\lambda = 1$. Then the solution to the initial value problem is $y=e^t$. We solve the ODE with the adaptive step method Filtered-IE23 and constant step method IE-Pre-Post-3 on the interval $0 \leq t \leq 2$. The results are in the table below.
\begin{table}[ht]
    \centering
    \tiny
    \begin{tabular}{|l|l|l|l|l|l|}
    \hline
        Adaptive Tolerance Setting & Filtered-IE23 Steps & Filtered-IE23 Error \\ \hline
        0.005 & 200   & 1.54956E-05 \\ \hline
        0.00025 & 2000  & 1.59584E-08 \\ \hline \hline
        & IE-Pre-Post-3 Steps & IE-Pre-Post-3 Error \\ \hline
        & 200 or $\Delta t$ = 0.01 & 1.55776E-05 \\ \hline
        & 2000 or $\Delta t$ = 0.001 & 1.59638E-08 \\ \hline
    \end{tabular}
    \caption{Filtered-IE23 versus IE-Pre-Post-3 on the Model Problem}
    \label{numerics:adaptive-constant-test}
\end{table}

In \Cref{numerics:adaptive-constant-test}, we see that Filtered-IE23 achieves a similar error at its final step to IE-Pre-Post-3 for the same number of steps. This is expected since they are both 3rd order methods and the problem is not particularly difficult. However, this gives evidence that the implementation is correct and is consistent with the method derivation. We will see more challenging problems in the following subsections.

For further reference, \Cref{numerics:IE-Pre-Post-3} and \Cref{numerics:IE-Pre-2} show the convergence of the constant step filtered methods run on the model problem. The successive error ratios are computed while doubling the number of steps (i.e. halving the time step). The base two $\log$ of the error ratios gives an estimate of the order of the method. We see the order estimate of IE-Pre-Post-3 approach 3 and the order estimate of IE-Pre-2 approach 2, as expected.

\begin{table}[ht]
    \tiny
    \centering
    \begin{tabular}{|l|l|l|l|}
    \hline
        Steps & Error & Error Ratio & Order  \\ \hline
        40   & 1.74388E-03 & 7.46631 & 2.90040 \\ \hline
        80   & 2.33566E-04 & 7.72961 & 2.95040 \\ \hline
        160  & 3.02170E-05 & 7.86411 & 2.97528 \\ \hline
        320  & 3.84240E-06 & 7.93191 & 2.98767 \\ \hline
        640  & 4.84422E-07 & 7.96608 & 2.99387 \\ \hline
        1280 & 6.08106E-08 & 7.98530 & 2.99735 \\ \hline
        2560 & 7.61532E-09 & 8.00833 & 3.00150 \\ \hline
    \end{tabular}
    \caption{IE-Pre-Post-3 Convergence Test}
    \label{numerics:IE-Pre-Post-3}
\end{table}
\begin{table}[ht]
    \centering
    \tiny
    \begin{tabular}{|l|l|l|l|}
    \hline
        Steps & Error & Error Ratio & Order  \\ \hline
        40    & 5.08667E-02 & 3.88217 & 1.95686 \\ \hline
        80   & 1.31026E-02 & 3.93307 & 1.97566 \\ \hline
        160  & 3.33140E-03 & 3.96436 & 1.98709 \\ \hline
        320  & 8.40338E-04 & 3.98162 & 1.99335 \\ \hline
        640  & 2.11054E-04 & 3.99066 & 1.99663 \\ \hline
        1280 & 5.28871E-05 & 3.99529 & 1.99830 \\ \hline
        2560 & 1.32373E-05 & 3.99764 & 1.99915 \\ \hline
    \end{tabular}
    \caption{IE-Pre-2 Convergence Test}
    \label{numerics:IE-Pre-2}
\end{table}

\subsection{Quasi-Periodic Oscillation ODE}
Next, we test Filtered-IE23 and the constant step methods on an ODE that has quasi-periodic oscillations. It is stated as
\begin{align*}
    &x'''' + (\pi^2 + 1)x'' + \pi^2 x = 0,\\[2ex]
    &x(0)=2,\,x'(0)=0,\,x''(0)=-(1+\pi^2),\,x'''(0)=0.
\end{align*}
The exact solution is $x(t) = \cos t + \cos \pi t$. This solution is the sum of periodic functions with incommensurable periods which makes it quasi-periodic \cite{corduneanu:1989}. The solution plot is presented in \Cref{fig:quasi-periodic}.

\begin{figure}[ht]
  \centering
  \includegraphics[width=0.45\textwidth]{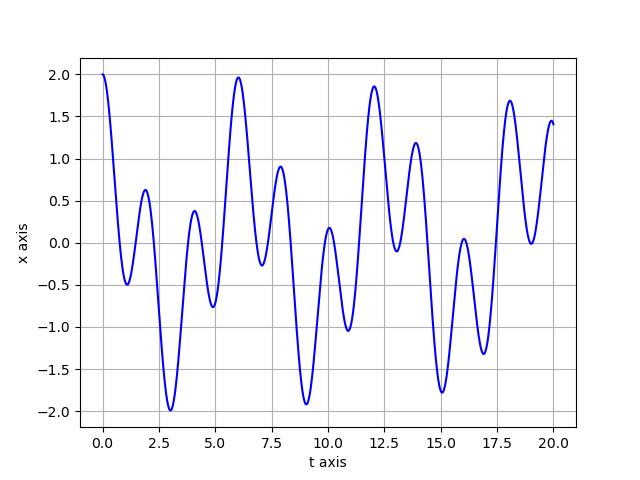}
  \caption{Quasi-Periodic Oscillation ODE Solution}
  \label{fig:quasi-periodic}
\end{figure}

We run the constant step methods, IE-Pre-2 and IE-Pre-Post-3, with step sizes $\Delta t = 0.1, 0.05, 0.01$. We set the tolerance to $0.0075$ for Filtered-IE23. The results are presented in \Cref{fig:quasiperiodic-plots-step-01} and \Cref{fig:quasiperiodic-plots-step-02}. Filtered-IE23 is able to capture the solution with its adaptive time stepping and the constant step filtered methods approach the true solution as we decrease the step size, as expected. Errors between the adaptive step method Filtered-IE23 and the higher order constant step method IE-Pre-Post-3 are presented in \Cref{numerics:adaptive-constant-quasi}. The errors are comparable between the adaptive and constant step method around 2000 steps. This suggests that the error estimator is a reliable way to have the adaptive method home in on the solution without the need to try many different constant step settings when running a computational tool based on this method.

\begin{table}[ht]
    \centering
    \tiny
    \begin{tabular}{|l|l|l|l|l|l|}
    \hline
        Adaptive Tolerance Setting & Filtered-IE23 Steps & Filtered-IE23 Error \\ \hline
        0.0075 & 2000   & 2.11559E-03 \\ \hline \hline
        & IE-Pre-Post-3 Steps & IE-Pre-Post-3 Error \\ \hline
        & 200 or $\Delta t$ = 0.1 & 1.98829E+00 \\ \hline
        & 400 or $\Delta t$ = 0.05 & 2.86552E-01 \\ \hline
        & 2000 or $\Delta t$ = 0.01 & 2.11669E-03 \\ \hline
    \end{tabular}
    \caption{Filtered-IE23 versus IE-Pre-Post-3 on the Quasi-Periodic ODE}
    \label{numerics:adaptive-constant-quasi}
\end{table}

\begin{figure}[ht]
     \centering
     \includegraphics[width=0.45\textwidth]{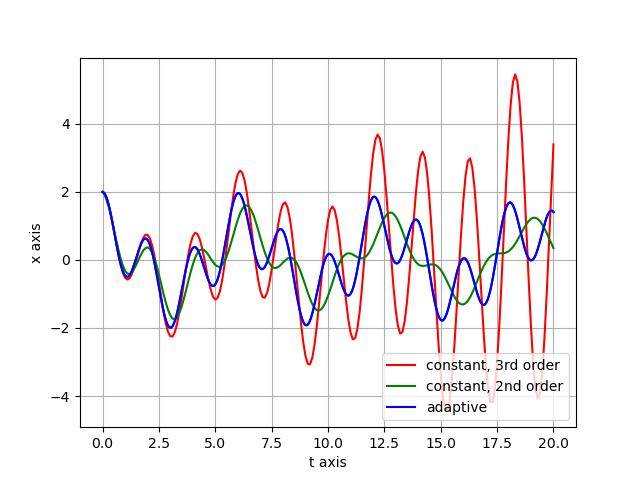}
     \caption{Constant Step Methods set to $\Delta t = 0.1$}
      \label{fig:quasiperiodic-plots-step-01}
\end{figure}
 \begin{figure}[ht]
     \centering
     \includegraphics[width=0.45\textwidth]{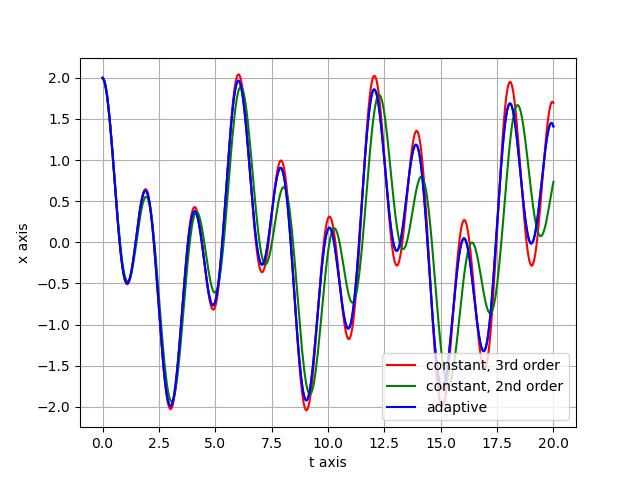}
     \caption{Constant Step Methods set to $\Delta t = 0.05$}
     \label{fig:quasiperiodic-plots-step-02}
\end{figure}

\subsection{Model Problem Analog}
Now we consider $y'=\lambda(t) y$. If we define $\lambda$ to be a function of $t$, then we have a nonautonomous analog of the model problem. It is useful to test the nonautonomous stability of a method on problems of this nature. We use
\begin{align*}
    y' &= (\gamma - 2t) y,\\
    y(0) &= 1.
\end{align*}
The solution is $y = e^{\gamma t - t^2}$. Increasing the parameter $\gamma$ increases the stiffness of the problem. We choose $\gamma = 1, 3, 5, 5.7, 6$. The results are compiled in \Cref{numerics:model-problem-analog}. Plots for $\gamma=5$ and $\gamma=6$ are in \Cref{fig:model-analog-plot-gamma-5} and \Cref{fig:model-analog-plot-gamma-6}
\begin{table}[ht]
    \centering
    \tiny
    \begin{tabular}{|l|l|l|l|}
    \hline
        $\gamma$ Value & Tolerance Setting & Initial Step Setting & Filtered-IE23 Error \\ \hline
        1.0 & 2.5e-05 & 1e-05  & 1.26305E-06 \\ \hline
        3.0 & 2.5e-05 & 1e-05  & 2.34021E-07  \\ \hline
        5.0 & 0.00025 & 0.0001  & 3.49478E-06 \\ \hline
        5.7 & 0.00025 & 0.0001  & 2.43668E-06 \\ \hline
        6.0 & 0.0005 & 0.0001 & 3.34943E-06 \\ \hline
    \end{tabular}
    \caption{Filtered-IE23 on the Model Problem Analog}
    \label{numerics:model-problem-analog}
\end{table}
\begin{figure}[ht]
     \centering
     \includegraphics[width=0.45\textwidth]{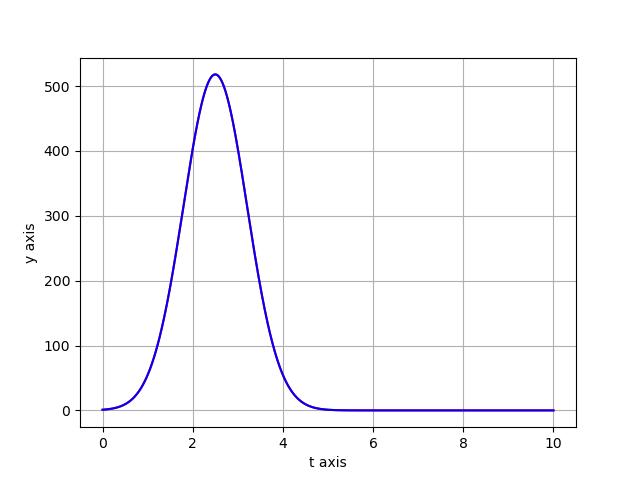}
     \caption{Model ODE Analog Solution $\gamma=5$}
     \label{fig:model-analog-plot-gamma-5}
\end{figure}
\begin{figure}[ht]
     \centering
     \includegraphics[width=0.45\textwidth]{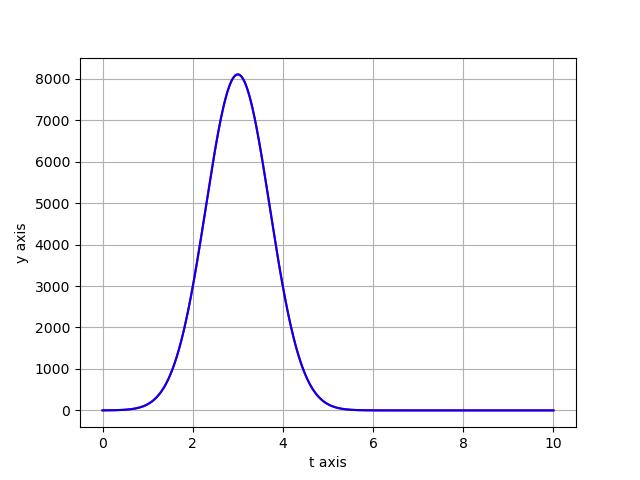}
     \caption{Model ODE Analog Solution $\gamma=6$}
     \label{fig:model-analog-plot-gamma-6}
\end{figure}

The results of this test suggest that the adaptive step method Filtered-IE23 may have favorable stability properties for nonautonomous problems. Additionally, the adaptive method can handle various levels of stiffness as evidenced by the increasing of the parameter $\gamma$.

\subsection{van der Pol Oscillator}
Here we test Filtered-IE23 on the van der Pol Oscillator, presented as
\begin{align*}
    x'' - \mu ( 1 - x^2 ) x' + x = 0.
\end{align*}
The van der Pol Oscillator satisfies Lienard's Theorem which ensures there is a stable limit cycle \cite{kanamaru:2022}. The greater the parameter $\mu$, the greater the damping effect. We choose $\mu = 1, 2, 5, 10, 100, 200$. The plots for $\mu=100$ and $\mu=200$ are presented in \Cref{fig:vanderpol-plots-100} and \Cref{fig:vanderpol-plots-200}.

\begin{figure}[ht]
     \centering
     \includegraphics[width=0.45\textwidth]{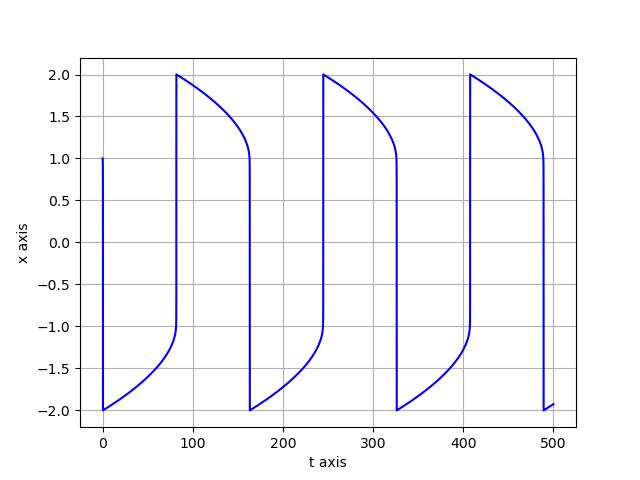}
     \caption{Filtered-IE23 Plots Solving van der Pol Oscillator $\mu=100$}
     \label{fig:vanderpol-plots-100}
\end{figure}
\begin{figure}[ht]
     \centering
     \includegraphics[width=0.45\textwidth]{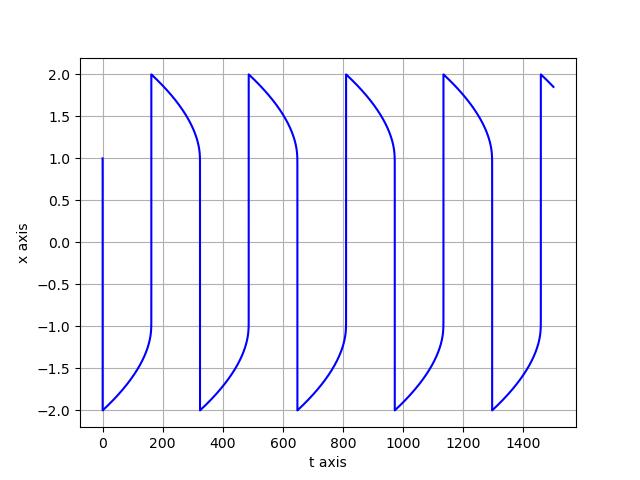}
     \caption{Filtered-IE23 Plots Solving van der Pol Oscillator $\mu=200$}
     \label{fig:vanderpol-plots-200}
\end{figure}

Larger values of $\mu$ cause difficulty for constant step methods except for relatively small step sizes.  Evidence of this can be seen in the sharpness of the transition at the local minima and maxima in the graphs of $\mu=100$ and $\mu = 200$. The adapative method Filtered-IE23 is more accurate at less cost than the constant step methods, which struggle navigating the sharp corners.

As we do not have an analytic solution, we use scipy's solve\_ivp set to RK45 to compare to Filtered-IE23's solution. The plots for the adaptive method Filtered-IE23 and RK45 nearly overlap. The differences in the $x$ values at the last step between the two methods are captured in \Cref{numerics:vanderpol-ie23-rk45}. The results of the test show that the adaptive step method Filtered-IE23 manages the sharp transitions well and maintains accuracy.
\begin{table}[ht]
    \centering
    \tiny
    \begin{tabular}{|l|l|l|l|l|}
    \hline
        Parameter $\mu$ & Final t value & Final x value FIE23 & Final x value RK45  & Method Difference \\ \hline
        1.0 & 50.0 & -1.61024 &-1.57738 & 3.28643E-02 \\ \hline
        2.0 & 50.0 & -1.08192 &-.980039 & 1.01883E-01 \\ \hline
        5.0 & 100.0 & 1.95843 &1.97366 & 1.52230E-02 \\ \hline
        10.0 & 200.0 & -1.18267 &-1.03524 & 1.47425E-01 \\ \hline
        100.0 & 500.0 & -1.92649 &-1.91881 & 7.67690E-03 \\ \hline
        200.0 & 1500.0 & 1.85147 &1.85743 & 5.96748E-03 \\ \hline
    \end{tabular}
    \caption{Filtered-IE23 vs RK45 on van der Pol}
    \label{numerics:vanderpol-ie23-rk45}
\end{table}

\section{Conclusions}
\label{sec:conclusions}

The constant step methods IE-Pre-2 and IE-Pre-Post-3 from \cite{decaria:2022} are evidence that filters can increase the order of Implicit Euler while still maintaining favorable stability properties. The work presented herein shows that the natural embedded pair within these methods can be used for effective adaptive stepping. We have found the pre-filter and post-filter to construct the adaptive extension of IE-Pre-2 and IE-Pre-Post-3, and performed initial testing. Despite the choice of the simplest controller, naive halving and doubling, the adaptive step method Filtered-IE23 performed well on the numerical tests run thus far. In situations where the constant step methods and the adaptive step extension perform similarly, the error estimator can be relied on to guide the adaptive method to the solution. Additionally, Filtered-IE23 has shown that it can navigate a problem where the constant step methods struggle. This is promising for more finely tuned implementations of the method.

The problem of implementing higher accuracy methods in legacy codes leads to interesting questions in numerical analysis. One approach is to add time filters, as herein, but others would be welcome. Next steps for research on Filtered-IE23 are fine tuning the implementation, further numerical testing, and analysis on the open question of stability.

\section{Data Availability Statement}
Source code for a Python implementation is available at
\url{https://github.com/stevemcgov/filtered-ie23} under a standard MIT License. Test scripts to reproduce the results are included in the code repository. Consult the repository README for setup instructions.

\section{Competing Interest Declaration}
There is no competing interest.

\section{Ethical Statement}
There is no conflict of interest.

\section{Funding Declaration}
There was no funding.

\section{Author Contribution Declaration}
There is only one author.

\bibliographystyle{siamplain}
\bibliography{references}
\end{document}